\documentclass{amsart}
\usepackage[utf8]{inputenc}
\usepackage{amsmath, amssymb, amsthm}
\usepackage{enumitem}
\usepackage{comment}
\usepackage[english]{babel}
\usepackage[T1]{fontenc}
\usepackage{graphicx}
\usepackage{wrapfig}
\usepackage{xcolor}

\def\bN{\mathbb{N}}

\def\bR{\mathbb{R}}

\def\bS{\mathbb{S}}

\def\cA{\mathcal{A}}
\def\cB{\mathcal{B}}
\def\cF{\mathcal{F}}
\def\cE{\mathcal{E}}
\def\cP{\mathcal{P}}

\def\cK{\mathcal{K}_{\mathcal{F}}}
\def\cV{\mathcal{V}_{\mathcal{F}}}

\def\cS{\mathcal{S}_{\nu}}

\def\homS{\mbox{Hom}(\bS^1 )}

\def\CO{\mbox{CO}}

\newtheorem{theorem}{Theorem}
\newtheorem{corollary}{Corollary}
\newtheorem{proposition}{Proposition}
\newtheorem{lemma}{Lemma}
\newtheorem{example}{Example}
\newtheorem{remark}{Remark}

\DeclareMathSymbol{\emptyset}{\mathord}{AMSb}{"3F}
\title[RDS on the circle without a finite orbit]{ Random Dynamical Systems on the circle without a finite orbit}

\author[D. Malicet]{Dominique Malicet}
\address{Université Gustave Eiffel, Laboratoire d’analyse et de mathématiques appliquées (LAMA), 5 Bd Descartes, 77420 Champs-sur-Marne, Francia}
\email{dominique.malicet@univ-eiffel.fr}

\author[G. Salcedo]{Graccyela Salcedo}
\address{Centre de Physique Théorique, CNRS, Institut Polytechnique de Paris, Palaiseau, France.}
\email{graccyela.salcedo@polytechnique.edu}

\begin{document}

\begin{abstract}
In this paper, we study Random Dynamical Systems (RDSs) of homeomorphisms on the circle without a finite orbit. We characterize the topological dynamics of the associated semigroup by identifying the existence of invariant sets which are finite unions of intervals. We describe the accumulation points of the average orbit of the transfer operator. For each ergodic stationary measure, we demonstrate interesting properties of its weight function on the circle. Relationships between the minimal sets of an RDS and its inverse RDS are also established.
\end{abstract}

\begin{thanks}
{This study was financed by the grant 2022/10341-9 Sao Paulo Research Foundation (FAPESP); and the Center of Excellence “Dynamics, Mathematical Analysis and Artificial Intelligence” at Nicolaus Copernicus University in Toruń.}
\end{thanks}

\keywords{random dynamical systems, local contraction, minimality, proximality, stationary measure}

\maketitle

\section{Introduction}

Given a set $\cF\subset\homS$, we study the \emph{Random Dynamical System (RDS)} associated to $\cF$ and the action on the circle of the semigroup generated by $\cF$. The action of the subgroup generated by $\cF$ has been extensively studied, see for example \cite{Anton:1984}, \cite{DerKleNav:07}, and \cite{Nav:11}. Here we assume that there is no finite set invariant under all maps in $\cF$. By \cite{Mal:17}, there exists a number $d=d_\cF\in\bN$ such that
\begin{itemize}
    \item the semigroup $G_\cF$ generated by $\cF$ has $d$ closed minimal sets, and
    \item each RDS associated to $\cF$ has $d$ ergodic stationary measures.
\end{itemize}
This is an immediate consequence of the trichotomy established in \cite[Corollary 2.2]{Mal:17}. Since we are assuming the absence of invariant finite sets, i.e., the absence of finite $G_\cF$-orbits; we have that one (and only one) of the following possibilities
occurs:
\begin{enumerate}
    \item[(i)] The semigroup $G_\cF$ does not preserve a probability measure. Hence, each RDS associated to $\cF$ has the local contraction property. By \cite[Proposition 4.8]{Mal:17}, has a finite number $d$ of ergodic stationary probability measures. Their respective topological supports are pairwise disjoints and exactly the minimal invariant compacts of $G_\cF$.
    \item[(ii)] The semigroup $G_\cF$ is semiconjugated\footnote{See \cite[Definition 2.1.3]{Nav:11}, for a definition of semiconjugation between semigroups.} to a semigroup of the compact group of the isometries of the circle acting minimally on $\bS^1$. Therefore, the semigroup $G_\cF$ has a unique closed minimal set, and any RDS associated to $\cF$ has a unique stationary measure, which is the invariant one.
\end{enumerate}

A stationary measure for an RDS reflects the long-term distribution of the orbits under the random system, capturing the notion of probabilistic equilibrium. Due to the compactness of $\bS^1$ and the continuity of the maps in the RDS, we have that for any initial point $x\in\bS^1$, the average of the transfer operator in the orbit of $x$ converges to a stationary measure $m_x$ \cite{Fur:63}.
We describe the set $\{m_x\colon x\in\bS^1\}$, i.e., the stationary measures that can be generated as the limit of an orbit of the transfer operator. According to the classical Birkhoff's Theorem, each ergodic stationary measure is generated by at least one orbit. The continuity of the map $x\mapsto m_x$ was already established in \cite{Mal:17}. Here, we characterize this map on the circle beyond its continuity. To do this, we use the topological structure of $\bS^1$ and show the existence of finite families of closed intervals, such that the union over any one of these families is invariant under every map in $\cF$. The existence of a family of intervals with that invariance property was established for $\cF\subset\text{Hom}^+(\bS^1)$ finite and $G_\cF$ having at least two closed minimal sets, see \cite[Theorem 1.2]{2018:KlepKudOku}.

To introduce our results, we first formalize the topological concepts mentioned above.
For $\cF\subset\bS^1$, let $G_\cF$ be the semigroup generated by $\cF$, i.e., $G_\cF$ is the smallest semigroup of $\homS$ containing $\cF$. Here, we assume that $G_\cF$ has no finite orbit, i.e., for every $x\in\bS^1$ the set
\[
\mathcal{O}_\cF(x)= \{f(x)\colon f\in G_\cF\}
\]
is not finite. A nonempty set $K\subset\bS^1$ is called \emph{$\cF$-invariant} if $f(K)\subset K$ for all $f\in\cF$. A nonempty closed set $K\subset\bS^1$ is called \emph{$\cF$-invariant minimal} if it is $\cF$-invariant and there is not a nonempty closed set $K'\subsetneq K$ which is $\cF$-invariant. 

\begin{theorem}\label{newteo:1}
   Let $\mathcal{F}\subset Hom(S^1)$ such that the generated semigroup
$G_{\mathcal{F}}$  has no finite orbit.  
Let $K_1$,...,$K_d$ be the
invariant minimal sets. Let us assume that $d\geq 2$. Then there
exist subsets $L_1,\ldots, L_d$ of $\bS^1$ such that
\begin{itemize}
    \item $L_i$ is a finite union of closed intervals with extremities in $K_i$,
    \item $L_i$ is $\cF$-invariant and contains $K_i$,
    \item the action of $G_{\mathcal{F}}$ on each connected component is proximal, and
    \item the sets $L_i$ are pairwise disjoint.
\end{itemize}
Moreover, if $\mathcal{F}\subset Hom^+(S^1)$, the sets $L_i$ has the
same number of connected components.
 \end{theorem}

Now to state the rest of our results, let us introduce some probabilistic concepts involved.
 For $\cF\subset\homS$, a \emph{Random Dynamical System (RDS)}  associated to $\cF$ is given by a probability measure $\nu$ on $\homS$ such that the support of $\nu$ is exactly $\cF$. To be more specific, we say that $(\cF,\nu)$ is an RDS.
The corresponding product probability space $(\cF^\bN, \nu^\bN)$ is ergodic invariant under the shift map $\sigma$. Define the \emph{random walk} $\omega \mapsto (f^n_\omega)$ on the semigroup $G_\cF$ generated by $\cF$ as follows:
\[
f^n_\omega = f_{n} \circ \cdots \circ f_1,
\]
with the implicit notation  $\omega = (f_n)_{n \in \mathbb{N}}\in \cF^\bN$. We use the notation $f^0_{\omega}=\textit{id}_{\bS^1}$. 
Consider the \emph{skew-shift map} 
\[
T: (\omega, x) \mapsto (\sigma\omega, f_1(x))
\]
on $\cF^\bN \times \bS^1$. A probability measure $\mu$ on $\bS^1$ is called \emph{stationary} if the product measure $\nu^\bN\otimes\mu$ is $T$-invariant. When $\nu^\bN\otimes\mu$ is $T$-ergodic we say that $\mu$ is an \emph{ergodic stationary} measure for $(\cF, \nu)$.

As we discuss in Section \ref{sec:w-maps}, by \cite{Fur:63} and \cite{Mal:17}, for an RDS $(\cF,\nu)$ of circle homeomorphisms without finite orbits, we have that for every $x$, the sequence 
\begin{equation}\label{eq:seqaverage}
    \left(\frac{1}{N} \sum_{n=0}^{N-1}\int_{\cF^\bN}\delta_{f_\omega^n(x)}\,d\nu^\bN(\omega)\right)_{n\in\bN}
\end{equation}
converges in the weak-$*$ topology to a probability measure $m_x$ on $\bS^1$, which is stationary for $(\cF,\nu)$. Moreover, the map $x\mapsto m_x$ is continuous on $\bS^1$. 

\begin{theorem}\label{teo:01}
    Let $(\cF,\nu)$ be an RDS of circle homeomorphisms without finite orbits. Let $d$ denote the number of ergodic stationary measures $\mu_1,\ldots, \mu_{ d }$. For $x\in\bS^1$, consider the limit measure $m_x$ on $\bS^1$ in the weak-$*$ topology of the sequence in \eqref{eq:seqaverage}. Also, consider the continuous maps $u_1,\ldots,u_d\colon \bS^1\to [0,1]$ such that 
    \[
    m_x=\sum_{i=1}^d u_i(x)\mu_i.
    \]
    Then, for every $x\in\bS^1$ 
    \[
    \#\left\{j\in\{1,\ldots,d\}:u_j(x)>0\right\}\leq 2,
    \]
    and for every $i\in\{1,\ldots,d\}$
    \[
    \#\left\{j\in\{1,\ldots,d\}\backslash\{i\}:\{u_i>0\}\cap \{u_j>0\}\neq\emptyset\right\}\leq 2.
    \]
    Moreover, for every $i\in\{1,\ldots,d\}$
    \begin{enumerate}[label= (\roman*)]
        \item both sets $\{u_i=0\}$ and $\{u_i=1\}$ are invariant by all $f\in\cF$ and can be written as a finite union of closed intervals,
        \item $u_i$ is monotone on any interval $I\subset \{0<u_i<1\}$, and
        \item if $u_i|_I\equiv C$ for some interval $I\subset\bS^1$ and $C\in[0,1]$, then for every $f\in\cF$ there exists $c_f\in[0,1]$ such that $u_i|_{f(I)}\equiv c_f$.
    \end{enumerate}
\end{theorem}

For $\cF\subset\mbox{Hom}(\bS^1 )$, consider $\cF^{-}=  \{f^{-1}\colon f\in\cF\}$. Also, consider $\nu^{-}$, the probability measure supported on $\cF^{-}$ naturally induced by $\nu$. We refer to $(\cF^{-},\nu^{-})$ as the \emph{inverse RDS} of $(\cF,\nu)$. In Example \ref{ex1}, we show an RDS where the number of ergodic measures differs from that of the inverse RDS. The following theorem is proved in Section \ref{sec:d+ and d-}.

\begin{theorem}\label{teo2}
    Let $(\cF,\nu)$ be an RDS of circle homeomorphisms without finite orbits and consider its respective inverse RDS $(\cF^{-},\nu^{-})$. Let $d_+$ and $d_-$ be the numbers of ergodic stationary measures associated to $(\cF,\nu)$ and $(\cF^{-},\nu^{-})$ respectively. Then,
    \[
    |d_+-d_-|\leq 1.
    \]
    Moreover, if $\cF\subset\mbox{Hom}^+(\bS^1 )$, then $d_+=d_-$.
\end{theorem}

Since, $\cF$ has a finite invariant set if and only if $\cF^-$ has a finite invariant set, we have the following topological consequence of the number of minimal sets.

\begin{corollary}
    Let $\cF\subset\mbox{Hom}(\bS^1 )$ and $\cF^{-}= \{f^{-1}\colon f\in\cF\}$. There is no finite invariant set under each map in $\cF$ simultaneously. Let $d_+$ and $d_-$ be the numbers of $\cF$-minimal sets and $\cF^-$-minimal sets respectively. Then,
    \[
    |d_+-d_-|\leq 1.
    \]
    Moreover, if $\cF\subset\mbox{Hom}^+(\bS^1 )$, then $d_+=d_-$.
\end{corollary}

This paper is organized as follows: We begin by discussing the topological properties of the minimal sets of the semigroup $G_\cF$ in Section \ref{sec:topprop}.
 In Section \ref{sec:w-maps}, we demonstrate part of Theorem \ref{teo:01}, which can be compared with Theorem \ref{teo:000000004}. In Section \ref{sec:neighbors}, we present results used to conclude the proof of Theorem  \ref{teo:01}. We also introduce the concepts of neighbor and path in Section \ref{sec:neighbors}. In Section \ref{sec:prosimal}, we discuss the concept of proximity and construct nontrivial intervals over which the action is proximal. A proof of Theorem \ref{newteo:1} is shown in Section \ref{sec:prosimal}. Finally, in Section \ref{sec:d+ and d-}, we study the inverse RDS and prove Theorem \ref{teo2}.

\section{Basic notations}
Let us introduce some preliminary notations, which we use throughout the text.

A \textit{circular order} of a finite enumerated set of distinct points \(\{x_1, x_2, \dots, x_n\} \subset \mathbb{S}^1\), $n>2$, is a permutation $\phi:\{1,2,\ldots,n\}\to\{1,2,\ldots,n\}$ such that: For any \(i,j,k\in\{1,2,\ldots,n\}\) with $\phi(i)<\phi(j)<\phi(k)$, the triplet $\langle x_{\phi(i)} ,x_{\phi(j)}, x_{\phi(k)}\rangle$ appears in the order \(x_{\phi(i)} \to x_{\phi(j)} \to x_{\phi(k)}\) when traversing the circle \(\mathbb{S}^1\) counterclockwise.
If $\phi$ and $\phi'$ are two circular orders of the same finite enumerated set \(\{x_1, x_2, \dots, x_n\}\), then $\phi$ is a translation of $\phi'$, that is, for some $i\in\{1,\ldots,n\}$, $\phi=\phi'+i\, (\text{mod } n)$. 

We say that a $n$-tuple \(\langle x_1, x_2, \dots, x_n\rangle \), \(n>2\), is \emph{circularly ordered} when the identity map over $\{1,2,\ldots,n\}$ is a circular order of $\{x_1, x_2, \dots, x_n\}$. Let $\CO$ denote the set of all $n$-tuples that are circularly ordered for all $n\geq 3$. In particular, a triplet  $\langle a,b,c\rangle$ is in $\CO$ if and only if it appears in the order \(a \to b \to c\) when traversing the circle \(\mathbb{S}^1\) counterclockwise.
Let us extend the concept of order to finite collections of sets. For $n\geq 3$ A $n$-tuple $\langle A_1,\ldots,A_n\rangle$ of pairwise disjoint subsets of $\bS^1$ is said \emph{circularly ordered} if, any set $\langle x_1,\ldots,x_n\rangle\in \CO$ whenever $x_i\in A_i$.

Given $a,b\in\bS^1$, $a\neq b$, the \emph{open interval}  $(a,b) $ is defined by
\[
(a,b)=\{x\in\bS^1\colon \langle a,x,b\rangle\in\CO\}.
\]
 The \emph{closed interval} $[a,b]$ is the union of the open interval $(a,b)$ with the pair \(\{a,b\}\). Here, we write \([a,a]=\{a\}\). These definitions follow the general notion of an interval (or, an arc) in the circle, identifying the circle as a counterclockwise real interval. For $a\neq b$ and $I\in\{[a,b],(a,b)\}$, we write $x<y$ for $x,y\in I$, $x\neq y$, if $\langle a,x,y,b\rangle\in\CO$. And, we write $x\leq y$ when $x=y$ or $x<y$.

Now, we introduce the concept of monotonicity for functions defined on intervals of the circle. Note that the image of an interval under a homeomorphism is another interval.
 Consider $a,b\in\bS^1$, $a\neq b$, and $I\in\{[a,b],(a,b)\}$. For $f\in\homS$,  we say that $f$ is \emph{monotone increasing} on $I$ if for any $x,y\in I$, $x\neq y$, such that $x<y$ then $f(x)\leq f(y)$. We call $f$ \emph{monotone decreasing} on an interval $I\in\{[a,b],(a,b)\}$ if for any $x,y\in I$, $x\neq y$, such that $x<y$ then $f(y)\leq f(x)$. And, we say that $f$ is \emph{monotone} if it is increasing or decreasing.

\section{Invariant minimal sets}\label{sec:topprop}

The results in this section are strictly topological and are established using the structure of the circle. Let $\cF\subset\homS$. Let $\cK$ be the collection of the $\cF$- invariant minimal sets. We assume $d=\#\cK<\infty$.

 If $d>1$. For $K\in\cK$, define $\cE_K$ as the family of the closed intervals such that
\begin{enumerate}
    \item[(a)] if $I=[x,y]\in\cE_K$ then $x,y\in K$,
    \item[(b)] $I\cap K'=\emptyset$ for all $K'\in\cK\backslash\{K\}$ and $I\in\cE_K$; and
    \item[(c)] for $I=[x,y]$ and $\hat I=[\hat x,\hat y]$ in $\cE_K$, we have$(\hat y,x)\cap K=(y,\hat x)\cap K =\emptyset$, and $(\hat y,x)\cap K'\neq \emptyset$, $(y,\hat x)\cap K'' \neq\emptyset$  for some $K',K''\in\cK$.
\end{enumerate}
When $\#\cK=1$, consider $\cE_{K}=\{\bS^1\}$.
We have
\begin{equation}\label{eq:kiinei}
    K\subset \bigcup_{I\in\mathcal{E}_K} I.
\end{equation}
In particular, note that $\cE_{K}$ is the \emph{maximal} collection of closed intervals with the following two properties:
\begin{enumerate}
    \item[(a')] If $I=[x,y]\in\cE_K$ then $x,y\in K$, and
    \item[(b')] $\left(\cup_{I\in\mathcal{E}_K} I\right)\cap K'=\emptyset$ for all $K'\in\cK\backslash\{K\}$.
\end{enumerate}
That is, if $\cE$ is a collection of closed intervals satisfying (a') and (b') (replacing $\cE_K$ by $\cE$) then 
\[
\bigcup_{I\in\cE} I\subset\bigcup_{I\in\mathcal{E}_K} I
\]

Let us establish the main result of this section which will be key to proving Theorem \ref{newteo:1} and Theorem \ref{teo:01}.

\begin{proposition}\label{Prop:000001}
    Assume that $\cF\subset\homS$ and $\cK$ is a finite set. Then, for every $x\in\bS^1$ there exist $K_1$ and $K_2$ in $\cK$ such that the orbit $\mathcal{O}_\cF(x)$ does not intercect any element of $\cK\backslash\{K_1,K_2\}$.
\end{proposition}

Note that Proposition \ref{Prop:000001} does not say that all the cluster points of the sequence $(f_\omega^n(x))_{n\in\bN}$ are contained in $K_1$ or $K_2$. Initially, this is not the case. The proposition only states that if a limit point is in one of the minimal sets, then that set must be $K_1$ or $K_2$.

Before Proving Proposition \ref{Prop:000001}, let us establish some topological results.

\begin{lemma}\label{lemma:0001}
    Assume that $\cF\subset\mbox{Hom}(\bS^1 )$ and $\cK$ is a finite set. Then, for all $K\in\cK$, the set $\cE_K$ is finite.
\end{lemma}
\begin{proof}
Assume that $\cE_K$ is infinite for some $K\in\cK$. By the compactness of $\bS^1$, there exists a sequence of closed intervals $(I_n)_{n\in\bN}$ in $\cE_K$ such that 
$\text{diam}(I_n)\to 0$ as $n\to\infty$. We can assume (possibly passing to a subsequence) that if $I_n=[a_n,\hat a_n]$ then the sequences $(a_n)_{n\in\bN}$ and $(\hat a_n)_{n\in\bN}$ of $K\subset\bS^1$ are convergent to a point $a\in K$. By the definition of the collection $\cE_K$, there is $K_n\in\cK\backslash\{K\}$ and $b_n\in K_{n}$ with $\langle a_n,b_n,a_{n+1}\rangle\in\CO$. Since $\cK$ is a finite set, we can find a subsequence $(n_k)_{k\in\bN}$ and $K'\in\cK\backslash\{K\}$ such that $K_{n_k}=K'$ for all $k\in\bN$. Then, $(b_{n_k})_{k\in\bN}$ is a sequence in $K'$ convergent to a point $a$ in $K$. Which contradicts the compactness of $K'$. Therefore, $\mathcal{E}_K $ is a finite set. 
\end{proof}

For $K\in\cK$, let us write $\ell_K=\#\cE_K$. Set 
\[
\mathcal{E}_K=\{I_{K,1}=[x_{K,1},y_{K,1}],\ldots,I_{K,\ell_K}=[x_{K,\ell_K},y_{K,\ell_K}]\}
\]
such that for $j\in\{1,\ldots,\ell_K\}$, $\langle x_{K,j},y_{K,j},x_{K,j+1}\rangle\in\CO$, where we are considering $\ell_K+1=1$.

\begin{lemma}\label{lemma:00002}
      Assume that $\cF\subset\mbox{Hom}(\bS^1 )$ and $\cK$ is a finite set. Then, for all $K\in\cK$, $j\in\{1,\ldots,\ell_K\}$ and $f\in\cF$ there is $k_j,r_j\in\{1,\ldots,\ell_K\}$ such that $f(I_{K,j})\subset I_{K,k_j}$ and $f(I_{K,r_j})\subset I_{K,j}$.
\end{lemma}
\begin{proof}
Consider $f\in \cF$ and $K\in\cK$. By \eqref{eq:kiinei} and the $\cF$-invariance of $K$,
\[
f(x_{K,j}),f(y_{K,j})\in f(K)\subset K\subset \bigcup_{I\in\mathcal{E}_K} I,
\]
for all $j=1,\ldots,\ell_K$. Case $\ell_K=1$, it is easy to conclude the lemma for $k_1=r_1=1$.

Assume $\ell_K>1$.
For $j=1,\ldots,\ell_K$, there exist $k_j^-,k_j^+\in \{1,\ldots,\ell_K\}$ such that $f(x_{K,j})\in I_{K,k_j^-}$ and $f(y_{K,j})\in I_{K,k_j^+}$. 
Note that by definition, for all 
\[
a,b\in\{x_{K,1},y_{K,1}, x_{K,2},y_{K,2},\ldots, x_{K,\ell_K},y_{K,\ell_K}\},
\]
with $b\neq y_{K,j}$, $a\neq x_{K,j}$ we have that both intervals $[a,x_{K,j}]$ and $[y_{K,j},b]$ intersect $\cup_{K'\in\cK\backslash\{K\}}K$. Because, $[y_{K,i},x_{K,j}]\subset [a,x_{K,j}]$ and $[y_{K,j},x_{K,i}]\subset [y_{K,j},b]$ for some $i=1,\ldots,\ell_K$. By $\cF$-invariance, both intervals $f([a,x_{K,j}])$ and $f([y_{K,j},b])$ also intersect $\cup_{K'\in\cK\backslash\{K\}}K$. In particular, for all $I\in\cE_K$, $f([a,x_{K,j}])\not\subset I$ and $f([y_{K,j},b])\not\subset I$. Therefore, for all $i,j\in\{1,\ldots,\ell_K\}$, $i\neq j$, 
\[
k_i^+,k_i^-\notin\{k_j^+,k_j^-\},
\]
and so,
\[
\{k_j^-\colon j=1,\ldots,\ell_K\}=\{k_j^+\colon j=1,\ldots,\ell_K\}=\{1,\ldots,\ell_K\}.
\]
So we must necessarily have equality $k_j^-=k_j^+$ for all $j=1,\ldots,\ell_K$. Moreover, each interval $f(I)$ (with $I\in\cE_K$), is contained in some $J\in\cE_K$, and vice versa each $J\in\cE_K$ contains some $f(I)$ (with $I\in\cE_K$). Therefore we also conclude the existence of $r_j$.
The lemma follows.
\end{proof}

\begin{lemma}\label{lemma:00003}
    Assume that $\cF\subset\mbox{Hom}(\bS^1 )$ and $\cK$ is a finite set. Let $K_1,K_2\in \cK$. Let $y_1\in K_1$ and $x_2\in K_2$ such that $(y_1,x_2)\cap K=\emptyset$ for all $K\in \cK$. Then, for every $f$ and $K\in \cK\backslash\{K_1,K_2\}$ we have $f([y_1,x_2])\cap K=\emptyset$.
\end{lemma}
\begin{proof}
    Consider $K_1,K_2\in \cK$, $y_1\in K_1$ and $x_2\in K_2$ such that $(y_1,x_2)\cap K=\emptyset$ for all $K\in \cK$. Let us proceed by contradiction. Given $f\in\cF$. If $f([y_1,x_2])\cap K\neq \emptyset$ for some $K\in \cK\backslash\{K_1,K_2\}$, then there exists $I\in\cE_K$ such that $I\subset f([y_1,x_2])$. By Lemma \ref{lemma:00002}, there exists $J\in\cE_K$ such that $f(J)\subset I\subset f([y_1,x_2])$. Since $f$ is a homeomorphism, $J\subset [y_1,x_2]$. Hence, $(y_1,x_2)\cap K\neq\emptyset$  which contradicts the hypotheses. The Lemma is proven.
\end{proof}
Now, we are ready to prove Proposition \ref{Prop:000001}.

\begin{proof}[Proof of Proposition \ref{Prop:000001}]
Let $x\in \bS^1$. If $x\in K$, for some $K\in\cK$, it is clear that for all $\omega\in\cF^\bN$, $f_{\omega}^n(x)\in K$ for all $n\in\bN$. Hence, the proposition follows by taking $K_1=K_2=K$. If $x\notin \cup_{K\in\cK} K$, there exist $K_1$, $K_2\in \cK$ and $y_1\in K_1$, $x_2\in K_2$, such that $(y_1,x)\cap K=(x,x_2)\cap K=\emptyset$, for all $K\in\cK$.
By Lemma \ref{lemma:00003}, we get that for all $f\in\cF$, $f(y_1,x_2)\cap K=\emptyset$ for all $K\in\cK\backslash\{K_{1},K_{2}\}$. Therefore, the cluster set of the sequence $(f_\omega^n(x))_{n\in\bN}$ does not intercect any $K\in \cK\backslash \{K_{1},K_{2}\}$. The proposition is proven.    
\end{proof}

\section{Weight maps}\label{sec:w-maps}

Throughout this section, $(\cF,\nu)$ is an RDS of circle homeomorphisms without finite orbits. Let $K_1,\ldots, K_{ d }$ be the respective topological supports of the ergodic stationary measures $\mu_1,\ldots,\mu_d$. 
In our context, $\cK=\{K_1,\ldots, K_{ d }\}$ is the collection of the closed $\cF$-invariant minimal sets.

The random process $(X_n)_{n\geq 0}$ defined by 
\[
X_n(\omega)=f_n(X_{n-1}), 
\]
with $\omega=(f_n)_{n\in\bN}\in \cF^\bN$ and $X_0$ being a random variable taking values in $\bS^1$, is a Markov chain on the probability space $(\cF^\bN,\nu^\bN)$ with state space $\bS^1$. Due to the compactness of $\bS^1$, we can apply \cite[Lemma 2.5]{Fur:63} to get that for every $x\in\bS^1$ and $\nu^\bN$-almost every $\omega\in\cF^\bN$ of weak-$*$ cluster values of the sequence of probability measures in \eqref{eq:seqaverage} is constituted of stationary probability measures of the RDS. In particular, when $d=1$, for each $x\in\bS^1$ the sequence in \eqref{eq:seqaverage} converges to the unique stationary measure in the weak-$*$ topology.

Define the \emph{weight maps} $u_i:\bS^1\to[0,1]$, $1\leq i \leq  d $, by
\begin{equation}\label{eq:def-u_i}
u_i(x)= \nu^\bN\left(\omega\in\cF^\bN\colon \mbox{cl}\left((f_\omega^n(x))_{n \geq 0}\right)=K_i\right),
\end{equation}
where for a set $A\subset\bS^1$, $\mbox{cl}(A)$ denotes the set of cluster values of $A$. By invariance of the sets in $\cK$, we have $u_i = \delta_{i,j}$ on $K_j$. By \cite[Proposition 4.9]{Mal:17}, for every $x\in\bS^1$ 
\begin{equation}\label{eq:sum-ui}
    \sum_{i=1}^d u_i(x)=1,
\end{equation}
and
\[
u_i(x)=\nu^\bN\left(\omega\in\cF^\bN\colon \frac{1}{N} \sum_{n=0}^{N-1}\delta_{f_\omega^n(x)}\mbox{ weakly-}* \mbox{ converges to } \mu_i\right),
\]
where $f_\omega^0(x)=x$. Moreover, for every $x\in \bS^1$ the sequence in \eqref{eq:seqaverage} converges in the weak-$*$ topology to the probability measure $m_x$ being as in Theorem \ref{teo:01}.

The \emph{transfer operator} $P$ associated to $(\cF,\nu)$ is defined on the space of the measurable bounded functions $\varphi:\bS^1\to \bR$ as follows
\[
P\varphi =  \int_{\cF}\varphi\circ f\,d\nu(f).
\]
In the context of Markov chains, the operator $P$ is known as the Markov operator of Markov chain $(X_n)_{n\geq 0}$. The iterates of $P$ are given by
\[
P^n\varphi =\int_{\cF^\bN}\varphi\circ f_{\omega}^n\,d\nu^\bN(\omega),
\]
so that the dynamic of $P$ represents the evolution of the law of the random variables $X_n$. Note that  for all $i\in\{1,\ldots,d\}$ the weight map $u_i$ defined in \eqref{eq:def-u_i} is $P$-invariant, i.e., $Pu_i=u_i$.
\begin{lemma}\label{lemma:ui=ouui=0}
    Let $(\cF,\nu)$ be an RDS of circle homeomorphisms without finite orbits.
    Let $d$ be the number of ergodic stationary measures and consider the weight maps $u_i$, $i=1,\ldots,d$ as in \eqref{eq:def-u_i}. Then, for all $i\in\{1,\ldots,d\}$, the sets
    \[
    \{u_i=0\} \quad \mbox{and}\quad \{u_i=1\}
    \]
    are $\cF$-invariant.
\end{lemma}
\begin{proof}
    Let us show the $\cF$-invariance of $\{u_i=1\}$. For $\{u_i=0\}$, the proof is analogous. We need to prove that for all $f\in\cF$
    \begin{align}\label{eq-01:teo1}
    f(\{x\in\bS^1\colon u_i(x)=1\})\subset \{x\in\bS^1\colon u_i(x)=1\}.
    \end{align}
    In fact, by the $P$-invariance, for $x$ such that $u_i(x)=1$ we get
    \begin{align*}
        1=u_i(x)&=Pu_i(x)=\int_\cF u_i(f(x))\,d\nu(f).
    \end{align*}
    Since $0\leq u_i\leq 1$ we get $u_i(f(x))=1$ for $\nu$-almost every $f\in\cF$. Using the continuity of $u_i$ and the fact that the support of $\nu$ is exactly $\cF$, we get $u_i(f(x))=1$ for all $f\in\cF$ and so \eqref{eq-01:teo1} holds.
\end{proof}

Let us establish a first result using what we discussed in Section \ref{sec:topprop}.

\begin{proposition}\label{prop1}
    Let $(\cF,\nu)$ be an RDS of circle homeomorphisms without finite orbits. 
    Let $d$ be the number of ergodic stationary measures. Let $\cK=\{K_1,\ldots, K_{ d }\}$ be the collection of the closed $\cF$-invariant minimal sets.
    Consider $u_i$ as in \eqref{eq:def-u_i} for $i=1,\ldots,d$. Then, for every $x\in\bS^1$ there exist $i_1,i_2\in\{1,\ldots,d\}$, $x_1\in K_{i_1}$, $x_2\in K_{i_2}$, such that $x\in[x_1,x_2]$, and for all $j\in\{1,\ldots,d\}\backslash\{i_1,i_2\}$ we have
    $u_j(x)=0$ and $[x_1,x_2]\cap K_j=\emptyset$.
\end{proposition}
\begin{proof}
    When $d=1$, the result it is evident by taking $i_1=i_2=1$ for all $x\in\bS^1$ because $u_1\equiv 1$. Now, consider $d\geq 2$.
    Since the collection $\cK=\{K_1,\ldots, K_{ d }\}$ of the closed $\cF$-invariant minimal sets is a finite set, we can apply Proposition \ref{Prop:000001} to get that for every $x\in\bS^1$ there exist $i_1,i_2\in\{1,\ldots,d\}$ such that for all $\omega\in \cF^\bN$ the cluster set of the sequence $(f_\omega^n(x))_{n\in\bN}$ does not intercect any element of $\cK\backslash\{K_{i_1},K_{i_2}\}$. In particular, we can find $x_1\in K_{i_1}$, $x_2\in K_{i_2}$, such that $x\in[x_1,x_2]$ and $[x_1,x_2]\cap K_j=\emptyset$ for all $j\in\{1,\ldots,d\}\backslash\{i_1,i_2\}$.  By \eqref{eq:sum-ui}, using the definition of the maps $u_i$ we get that $u_{i_1}(x)+u_{i_2}(x)=1$ (case $i_1\neq i_2$) or $u_{i_1}(x)=1$ (case $i_1= i_2$). Hence, $u_j(x)=0$ for all $j\in\{1,\ldots,d\}\backslash\{i_1,i_2\}$.
    \end{proof}

The main result of this section establishes more precisely the properties of the collections of closed intervals mentioned in Theorem \ref{teo:01}. Beyond this, we show that if $i$ and $j$ are neighbors (see Section \ref{sec:neighbors}, for a definition), then the functions $u_i$ and $u_j$ are monotone on any interval contained in $\{u_i>0\}\cap\{u_j>0\}$.

\begin{theorem}\label{teo:000000004}
    Let $(\cF,\nu)$ be an RDS of circle homeomorphisms without finite orbits. Let $d$ be the number of ergodic stationary measures. Consider $u_i$ as in \eqref{eq:def-u_i} for $i=1,\ldots,d$. Then, for all $i\in\{1,\ldots,d\}$ there exists a finite collection of disjoint closed intervals $\cA_i$ such that
    \begin{enumerate}[label= \arabic*)]
    \item[(a)] $K_i\subset \{x\in\bS^1\colon u_i(x)=1\}=\cup_{I\in\cA_i} I$;
        \item[(b)] for all $I,J\in\cA_i$ there exists $f\in\cF$ such that $f(I)\subset J$; and
        \item[(c)] for $I,J\in\cA_i$, $I\neq J$, if $x\in I$ and $y\in J$ then both intervals $[x,y]$ and $[y,x]$ of $\bS^1$ intersect the set $ \cup_{j\neq i}K_j$.
    \end{enumerate}
    Moreover, assuming $d\geq 2$, for $i,j\in\{1,\ldots,d\}$, $i\neq j$ and $I=[x_i,y_i]\in\cA_i$, $J=[x_j,y_j]\in\cA_j$, such that 
    \[
    (y_i,x_j)\cap \cup_{r=1}^d K_r=\emptyset,
    \]
    we have that $u_i$ is a monotone decreasing function on $[y_i,x_j]$ and $u_j$ is a monotone increasing function on $[y_i,x_j]$.
\end{theorem}

Note that  $(y_i,x_j)\subset \{u_i>0\}\cap\{u_j>0\}$. Moreover, the set $\{u_i>0\}\cap\{u_j>0\}$ is a finite union of open intervals with an extremal point in $\cup_{I\in\cA_i} I$ and the other extremal point in $\cup_{J\in\cA_j} J$.

\begin{proof}[Proof of Theorem \ref{teo:000000004}]
When $d=1$, the theorem it is clear taking $\cA_1=\{\bS^1\}$. Assume that $d\geq 2$.
    Let $i=1,\ldots,d$. Recall that $u_i:\bS^1\to[0,1]$ is a continuous function. Hence, $\{x\in\bS^1\colon u_i(x)=1\}$ is a closed subset of $\bS^1$ and there exists a collection $\cA_i $ of disjoint closed intervals such that
    \begin{equation}\label{set:ui=1}
        \{x\in\bS^1\colon u_i(x)=1\}=\cup_{I\in\cA_i}I.
    \end{equation}
    The continuity of $u_i$ implies that $\cup_{I\in\cA_i}I$ is a closed (and hence compact) subset of $\bS^1$.
    Let us show that $\cA_i $ is finite. Note that for $I=[x_I,y_I], J=[x_J,y_J]\in\cA_i$ such that 
    \[
    (y_I,x_J)\cap \left(\cup_{A\in\cA_i}A\right)=\emptyset,
    \]
    we have that $u_i(x)<1$ for all $x\in (y_I,x_J)$. Using \eqref{eq:sum-ui} and the definition of maps $u_j$, $j\in\{1,\ldots,d\}$, we get that there exists $j\neq i$ such that $u_j(x)=1-u_i(x)>0$ for all $x\in (y_I,x_J)$. By Lemma \ref{lemma:00003}, it is not difficult to see that $K_j\cap (y_I,x_J)\neq \emptyset$. Using that the set in \eqref{set:ui=1} is compact, since $\cA_i$ is an infinite collection of closed intervals of $\bS^1$, we can find a infinite sequence $(I_n=[x_n,\hat x_n])_{n\in\bN}$ such that 
    \[
    \lim_{n\to\infty}\text{diam}(I_n)=0\quad \text{and}\quad \lim_{n\to\infty}x_n=\lim_{n\to\infty}\hat x_n=x,
    \]
    for some $x\in \cup_{I\in\cA_i}I$. Also assume without loss of generality that for all $n\in\bN$ and $m\geq 2$, $\langle x_n,x_{n+1},\ldots,x_{n+m}\rangle\in\CO$. By definition of $\cA_i$, for all $n\in\bN$ there exist $i_n\in\{1,\ldots,d\}\backslash\{i\}$ and $y_n\in K_{i_n}\cap (\hat x_n,x_{n+1})$. Note that $(y_n)_{n\in\bN}$ is a sequence converging to $x$. Take a subsequence $(i_{n_k})_{k\in\bN}$ of $(i_n)_{n\in\bN}$ such that for some $j\in\{1,\ldots,d\}\backslash\{i\}$ and for all $k\in\bN$, $i_{n_k}=j$. Then, $(y_{n_k})_{k\in\bN}$ is a sequence in $K_j$ converging to a point in $K_i$, which contradicts the compactness of $K_j$. Consequently, $\cA_i$ is finite.

    Now, let us show items (a), (b), and (c). Item (a) is obvious by the definition of the map $u_i$ and the $\cF$-invariance of $K_i$. Using that $\cA_i$ is finite and proceeding analogously in the proof of Lemma \ref{lemma:00002}, we note that it is sufficient to show the $\cF$-invariance of $\{x\in\bS^1\colon u_i(x)=1\}$ to conclude item (b).  Finally, apply Lemma \ref{lemma:00003} and the fact that every $I\in\cA_i$ intercects $K_i$ to conclude item (c).

    To prove the second part of Theorem \ref{teo:000000004}, consider $i,j\in\{1,\ldots,d\}$, $i\neq j$ and $I=[x_i,y_i]\in\cA_i$, $J=[x_j,y_j]\in\cA_j$, such that 
    \[
    (y_i,x_j)\cap \cup_{r=1}^d K_r=\emptyset.
    \]
    By Proposition \ref{prop1}, for $x\in [y_i,x_j]$ we have $u_i(x)+u_j(x)=1$.
    Let us show that $u_i$ is a monotone decreasing function on $[y_i,x_j]$. Note that for $f\in\cF$ there exist $I(f)=[x_i(f),y_i(f)]\in\cA_i$, $J(f)=[x_j(f),y_j(f)]\in\cA_j$, such that $f(I)\subset I(f)$, $f(J)\subset J(f)$ and
    \[
    (y_i(f),x_j(f))\cap \cup_{r=1}^d K_r=\emptyset,\mbox{ or }(y_j(f),x_i(f))\cap \cup_{r=1}^d K_r=\emptyset.
    \]
    The last fact will depend on whether $f$ preserves orientation or not. Therefore, for $x\in [y_i,x_j]$ and $\omega \in\cF^\bN $ such that $\mbox{cl}\left((f_\omega^n(x))_{n \geq 0}\right)=K_i$, we get that for all $y\in [y_i,x]$ we also have $\mbox{cl}\left((f_\omega^n(y))_{n \geq 0}\right)=K_i$. Hence, 
    \[
    \{\omega\in\cF^\bN\colon \mbox{cl}\left((f_\omega^n(x))_{n \geq 0}\right)=K_i\}\subset \{\omega\in\cF^\bN\colon \mbox{cl}\left((f_\omega^n(y))_{n \geq 0}\right)=K_i\},
    \]
    for every $x\in [y_i,x_j]$ and $y\in [y_i,x]$. Consequently, for $x,y\in [y_i,x_j]$, $\langle y_i,x,y,x_j\rangle\in\CO$, we have that
    \[
    u_i(x)\leq u_i(y).
    \]
    Which implies the desired. The monotonicity of $u_j$ is a immediate consequence of $u_i+u_j=1$ on $[y_i,x_j]$.
\end{proof}
Let us prove item \textit{(iii)} in Theorem \ref{teo:01} as a corollary of Theorem \ref{teo:000000004}.
\begin{corollary}\label{cor:constants}
 Let $(\cF,\nu)$ be an RDS of circle homeomorphisms without finite orbits. Let $d$ be the number of ergodic stationary measures. Consider $u_i$ as in \eqref{eq:def-u_i} for $i=1,\ldots,d$. If $u_i|_I\equiv C$ for some interval $I\subset\bS^1$ and $C\in[0,1]$, then for every $f\in\cF$ there exists $c_f\in[0,1]$ such that $u_i|_{f(I)}\equiv c_f$.
\end{corollary}
\begin{proof}
    Consider $i\in\{1,\ldots,d\}$ and $I=[a,b]$ such that $u_i|_I\equiv C$ for some $C\in[0,1]$. By Lemma \ref{lemma:ui=ouui=0}, when $C\in\{0,1\}$, it is enough to consider $c_f=C$ for all $f\in\cF$. 
    
    Assume $C\in(0,1)$. Since $C<1$, $I\cap K_j=\emptyset$ for all $j\in\{1,\ldots,d\}$. Take $i_1,i_2\in\{1,\ldots,d\}$, $x_1\in K_{i_1}$, $x_2\in K_{i_2}$, such that $I\subset[x_1,x_2]$ and 
    \[
    (x_1,x_2)\cap(\cup_{\ell+1}^{d}K_\ell)=\emptyset.
    \]
    Since $C>0$, $i_1=i$ or $i_2=i$. Also, we have $i_1\neq i_2$. Consider $c_1,c_2>0$ such that $c_1+c_2=1$, $u_{i_1}|_I=c_1$ and $u_{i_2}|_I=c_2$. By the $P$-invarience of the weight maps, we get for $x,y\in I$, $x\leq y$ 
    \begin{align}\label{eq:cor-1-pfx=pfy}
        &\int_\cF \nu^\bN(\omega\in\cF^\bN\colon \mbox{cl}\left((f_\omega^n(f(x)))_{n \geq 0}\right)=K_{i_1})\,d\nu(f)\nonumber\\
        &=Pu_{i_1}(x)=u_{i_1}(x)\nonumber\\
        &=c_{i_1}\\
        &=u_{i_1}(y)=Pu_{i_1}(y)\nonumber\\
        &=\int_\cF \nu^\bN(\omega\in\cF^\bN\colon \mbox{cl}\left((f_\omega^n(f(y)))_{n \geq 0}\right)=K_{i_1})\,d\nu(f).\nonumber
    \end{align}
    By the continuity of the maps in $\cF$, for all $f\in\cF$
    \begin{equation}\label{eq:cor-1-i1}
    \{\omega\in\cF^\bN\colon \mbox{cl}\left((f_\omega^n(f(y)))_{n \geq 0}\right)=K_{i_1}\}\subset\{\omega\in\cF^\bN\colon \mbox{cl}\left((f_\omega^n(f(x)))_{n \geq 0}\right)=K_{i_1}\}.
    \end{equation}
    By \eqref{eq:cor-1-pfx=pfy} and \eqref{eq:cor-1-i1}, we get that for $\nu$-almost every $f\in\cF$ 
    \[
        \nu^\bN(\omega\in\cF^\bN\colon \mbox{cl}\left((f_\omega^n(f(x)))_{n \geq 0}\right)=K_{i_1})= \nu^\bN(\omega\in\cF^\bN\colon \mbox{cl}\left((f_\omega^n(f(y)))_{n \geq 0}\right)=K_{i_1}),
    \]
    i.e., $u_{i_1}(f(x))=u_{i_1}(f(y))$. By the continuity of $u_{i_1}$ and the fact that $\nu $ is supported on $\cF$, we conclude $u_{i_1}(f(x))=u_{i_1}(f(y))$ for all $f\in \cF$. Since $x$ and $y$ are arbitrary in $I$, we get $u_{i_1}|_{f(I)}=c_{1,f}$ for some $c_{1,f}\in[0,1]$. Analogously, we can show that for all $f\in\cF$ we have $u_{i_2}|_{f(I)}=c_{2,f}=1-c_{1,f}$. To conclude this demonstration, it only remains to consider $c_f=c_{k,f}$ for all $f\in\cF$ and $k\in\{1,2\}$ such that $i=i_k$.
\end{proof}

We say that a finite set $P=\{x_0,x_1,\ldots,x_{n_P}\}$, $n_P\geq 2$ is a \emph{partition} of $\bS^1$ if $x_{n_P}=x_0$, $x_i\neq x_{i+1}$ and $P$ appears in the order \(x_1 \to x_2 \to\cdots\to x_{n_P}\) when traversing the circle \(\mathbb{S}^1\) counterclockwise. In particular, when $n_P\geq 2$ we have that $\langle x_1,\ldots,x_{n_p}\rangle$ and $\langle x_0, x_1,\ldots,x_{n_p-1}\rangle$ are $n_P$-tuples of $\CO$. When $n_P=2$, $P$ is a set of the form $\{a,b,a\}$ with $a\neq b$. Let $\cP$ denote the set of the partitions $P$ of $\bS^1$. 

The \emph{total variation} of a real-valued function $\varphi:\bS^1\to \bR$, defined on $\bS^1$ is the quantity
\[
V(\varphi)=\sup_{P\in\cP}\sum_{i=0}^{n_P-1}\vert \varphi(x_i)-\varphi(x_{i+1})\vert.
\]
A function $\varphi:\bS^1\to \bR$ is said to be of \emph{bounded variation} (BV function) on $\bS^1$ if its total variation is finite.

From Theorem \ref{teo:000000004}, we can conclude the bounded variation of the weight maps using their continuity.
\begin{corollary}
    Let $(\cF,\nu)$ be an RDS of circle homeomorphisms without finite orbits. Let $d$ be the number of ergodic stationary measures. Then, for $i=1,\ldots,d$, the map $u_i$ defined \eqref{eq:def-u_i} has bounded variation.
\end{corollary}
\begin{proof}
    If $d=1$, the corollary follows easily because $u_1$ is constant, implying that $V(u_1)=0$.
    Assume $d\geq 2$ and take $i\in\{1,\ldots,d\}$. By Theorem \ref{teo:000000004}, we can find $\langle a_0,a_1,\ldots,a_m\rangle\in\CO$, with $m=4k-1$ for some $k\in\bN$, such that for $j=1,\ldots,k$
    \begin{itemize}
        \item $u_i=1$ on $[a_{4j-4},a_{4j-3}]$,
        \item $u_i$ is monotone decreasing on $[a_{4j-3},a_{4j-2}]$,
        \item $u_i=0$ on $[a_{4j-2},a_{4j-1}]$, and
        \item $u_i$ is monotone increasing on $[a_{4j-1},a_{4j}]$, where $a_{4k}=a_0$.
    \end{itemize}
    Therefore, for any $P\in \cP$ the sum 
    \[
    \sum_{i=0}^{n_P-1}\vert \varphi(x_i)-\varphi(x_{i+1})\vert
    \]
    can be separated into $m+1=4k$ sums with half of them being 0 and the other half (corresponding to the monotonicity intervals of $u_i$) being telescopic sums. Thus, using $0\leq u_i\leq 1$,
    \[
    \sum_{i=0}^{n_P-1}\vert \varphi(x_i)-\varphi(x_{i+1})\vert\leq 2k.
    \]
    Since $k$ does not depend on the partition $P$, we can conclude the bounded variation of $u_i$.
\end{proof}

\section{The neighbors}\label{sec:neighbors}

Use the same notation in Section \ref{sec:w-maps} in this section. For $i,j\in\{1,\ldots,d\}$ we say that $i$ and $j$ are neighbors if $\{u_i>0\}\cap\{u_j>0\}\neq\emptyset$.  See Lemma \ref{lemma:cV} for a nontrivial upper bound on the number of neighbors for $i$.
Define 
\begin{equation}\label{eq:defVF}
    \cV\subset\{1,\ldots,d\}\times\{1,\ldots,d\},
\end{equation}
as the set of pairs $(i,j)$ such that $i$ and $j$ are neighbors. That is, $(i,j)\in \cV$ if and only if there is $x\in\bS^1$ such that $u_i(x)>0$ and $u_j(x)>0$.

Define $\cS$ as the collection of stationary measures that are weak-$*$ limits of the sequence \eqref{eq:seqaverage} for some $x\in \bS^1$. 

Using the topological facts shown in Section \ref{sec:topprop}, let us prove the following results, which immediately imply the first part of Theorem \ref{teo:01}.
\begin{lemma}\label{corollary1:teo1}
    Let $(\cF,\nu)$ be an RDS of circle homeomorphisms without finite orbits.
    Let $d$ be the number of ergodic stationary measures $\mu_1,\ldots, \mu_{ d }$. Then,
    \[
    \cS=\left\{t\mu_i+(1-t)\mu_j\colon t\in[0,1],\,(i,j)\in\cV\right\}.
    \]
\end{lemma}
\begin{proof}
Take $x\in\bS^1$. By Proposition \ref{Prop:000001}, the orbit $\mathcal{O}_\cF(x) $ cannot intercept more than two minimal sets, so that for all $\omega \in \Omega $ the cluster of $(f_\omega^n(x))$ cannot intercept more than two minimal sets. Recall the sequence in \eqref{eq:seqaverage} converges to a stationary probability measure $m_x=\sum_{i=1}^d u_i(x)\mu_i$. Hence, $\#\{i\in\{1,\ldots,d\}\colon u_i(x)>0\}\leq 2$. In particular, $m_x\in\left\{t\mu_i+(1-t)\mu_j\colon t\in[0,1],\,(i,j)\in\cV\right\}$. Because of the arbitrariness of $x$,
$\cS\subset\left\{t\mu_i+(1-t)\mu_j\colon t\in[0,1],\,(i,j)\in\cV\right\}$, use  and \eqref{eq:sum-ui}. To conclude the equality, use the continuity of weight maps and the definition of $\cV$.
\end{proof}

Now, for $i\in\{1,\ldots,d\}$ consider 
    \begin{equation}\label{eq:def-Vi}
            V_i= \left\{j\in\{1,\ldots,d\}\colon j\neq i,\,(i,j)\in\cV\right\}.
    \end{equation}
    If $d=1$, it is clear that $V_1=\emptyset$. Let us establish the following property of the sets $V_i$, $i=1,\ldots,d$, which will be key to completing the second inequality in Theorem \ref{teo:01}.

    \begin{lemma}\label{lemma:cV}
    Let $(\cF,\nu)$ be an RDS of circle homeomorphisms without finite orbits. 
    Let $d$ be the number of ergodic stationary measures. Let $\cK=\{K_1,\ldots, K_{ d }\}$ be the collection of the closed $\cF$-invariant minimal sets.
    Consider $u_i$ as in \eqref{eq:def-u_i} for $i=1,\ldots,d$. If $d\geq 2$, then for all $i=1,\ldots,d$
    \begin{equation}\label{eq1:lemmacv}
        1\leq\#V_i\leq 2,
    \end{equation}
    where $V_i$ is as in \eqref{eq:def-Vi}. Moreover, 
    \begin{equation}\label{eq2:lemmacv}
        \#\{i\in \{1,\ldots,d\}\colon \#V_i=1\}\leq 2.
    \end{equation}
\end{lemma}
\begin{proof}[Proof of \eqref{eq1:lemmacv}]
     Assume $d\geq 2$. Fix $i=1,\ldots,d$. Using the fact that the collections $\cA_k$ defined in Theorem \ref{teo:000000004}, $k=1,\ldots d$ are finite and all their elements are closed sets, we obtain that there exists $j\neq i$ such that 
    \[
    \mbox{dist}(\cup_{I\in\cA_i}I,\cup_{J\in\cA_j}J)=\min_{k\neq i}\mbox{dist}(\cup_{I\in\cA_i}I,\cup_{J\in\cA_k}J).
    \]
    From the continuity of the weight maps, we can conclude that $j\in V_i$. Hence, $\# V_i\geq 1$.
    
    Now, let us show that $\# V_i\leq 2$. Take $j\in V_i$ and consider $I'=[a_i',b_i']\in\cA_i$ and $J'=[a_j',b_j']\in\cA_j$ such that 
    \begin{itemize}
        \item $(b_i',a_j')\cap A=\emptyset$ for all $A\in\cup_{k}\cA_k$, or
        \item $(b_j',a_i')\cap A=\emptyset$ for all $A\in\cup_{k}\cA_k$.
    \end{itemize}
    In any case,  by item (c) of Theorem \ref{teo:000000004}, we have that for every $I=[a_i,b_i]\in\cA_i$ there exists $f\in \cF$ and $J=[a_j,b_j]\in\cA_j$ such that $f(I')\subset I$ and $f(J')\subset J$. From
    Lemma \ref{lemma:00003}, we get
    \begin{itemize}
        \item $(b_i,a_j)\cap A=\emptyset$ for all $A\in\cup_{k}\cA_k$, or
        \item $(b_j,a_i)\cap A=\emptyset$ for all $A\in\cup_{k}\cA_k$.
    \end{itemize}
    We have just shown that for all $j\in V_i$ and $I=[a_i,b_i]\in\cA_i$ there exists $J=[a_j,b_j]\in V_j$ such that the two items above are satisfied. Using the topological structure of $\bS^1$, we conclude that $\# V_i\leq 2$ since each interval $I\in\cA_i$ has only two intervals in $\cup_{k\neq i}\cA_k$ as neighbors.
    \end{proof}

To prove \eqref{eq2:lemmacv}, let us first introduce the notion of $\cF$-paths over the index set $\{1,\ldots,d\}$. According to the conclusions in Theorem \eqref{teo:000000004}, we have a finite family of intervals on the circle, say $I_1,\ldots,I_k$, such that for each $\ell=1,\ldots,k$, there exists $r(\ell)\in\{1,\ldots,d\}$ for which $u_{r(\ell)}=1$ on $I_\ell$. Assume that the enumeration of these intervals is counterclockwise. Then, two different indices 
$i$ and $j$ are neighbors if, for some $\ell$, we have $r(\ell)=i$ and also, $r(\ell-1)=j$ or $r(l+1)=1$. With these concepts in mind, we define an $\cF$-path as any finite sequence on 
$\{1,\ldots,d\}$ of the form $[r(\ell),\ldots,r(\ell+n)]$, where $n\geq0$, $\ell\in\{1,\ldots,k\}$, and $r(k+m)=r(m)$, for $m\in\bN$. Let us formalize the definition of $\cF$-paths.

For $n\in\bN$ and $i_1,\ldots,i_n\in\{1,\ldots,d\}$, we say that $[i_1,\ldots,i_n]$ is a \emph{$\cF$-path} if for all $k=1,\ldots,n-1$
\begin{itemize}
    \item $i_k\neq i_{k+1}$, and
    \item there exist $I_k=[a_k,b_k]\in\cA_{i_k}$ and $I_{k+1}=[a_{k+1},b_{k+1}]\in\cA_{i_{k+1}}$ such that the interval $(b_k,a_{k+1})$ does not intercect any $I\in \cup_i \cA_i$.
\end{itemize}
Let us make some observations about $\cF$-paths.
\begin{remark}\label{rem:Fpaths}
Each $\cF$-path with length greater than $\#(\cup_i\cA_i)$, passes through all the indices in $\{1,\ldots,d\}$.
Also, note that $i$ and $j$ are neighbors if and only if $[i,j]$ or $[j,i]$ is a $\cF$-path. Further, if $d=1$ the unique $\cF$-path is $[1]$ and if $d=2$ the $\cF$-paths are as follows $[i_1,i_2,i_1,\ldots,i_1,i_2,i_1]$ or $[i_1,i_2,i_1,\ldots,i_2,i_1,i_2]$. 
\end{remark}

\begin{lemma}\label{lemma:not3paths}
    Let $(\cF,\nu)$ be an RDS of circle homeomorphisms without finite orbits. Let $d$ be the number of ergodic stationary measures. Then, for every $i,j,k,a,b,c\in\{1,\ldots,d\}$, with $i\neq j$, $j\neq k$ and $k\neq i$, at least one of the triples $[a,i,a]$, $[b,j,b]$ or $[c,k,c]$ is not a $\cF$-path.
\end{lemma}
\begin{proof}
Assume $d\geq 3$. For $d<3$ the conclusion is clear.  We proceed by reductio ad absurdum. Take $i,j,k,a,b,c\in\{1,\ldots,d\}$, with $i\neq j$, $j\neq k$ and $k\neq i$, such that $[a,i,a]$, $[b,j,b]$ and $[c,k,c]$ are $\cF$-paths, that is $\#V_i=\#V_j=\#V_k=1$. Using \eqref{eq1:lemmacv}, we can conclude that $\#V_a=2$. In fact, if $\#V_a=1$ then $[i,a,i,a,i]$ is an $\cF$-path and, in particular, for any $n\in\bN$ any $n$-tuple of the form $[i,a,i,a,\ldots,i]$ is an $\cF$-path, which is not possible by Remark \ref{rem:Fpaths} and the fact $d\geq 3$. Analogously, we have $\#V_a=\#V_b=\#V_c=2$. Using the circle structure, we get that for $n\in\bN$ and any two $\cF$-paths $[i_1,\ldots,i_n]$ and $[j_1,\ldots,j_n]$ with $i_1=j_n=a$ and $i_n=j_1=b$, we have $i_m=j_{n-m+1}$ for $m=1,\ldots,n$. Further, we have $i_m=k$ for some $m=1,\ldots,n$. Since we must conclude the same for $\cF$-paths starting at $a$ (or $b$) and ending at $c$. We can conclude that $a=b$. But this contradicts that $i\neq j$, since $\#V_a=2$. Therefore, at least one of the triples $[a,i,a]$, $[b,j,b]$ or $[c,k,c]$ is not a $\cF$-path.
\end{proof}

\begin{proof}[Proof of \eqref{eq2:lemmacv} in Lemma \ref{lemma:cV}]
    Since $\#V_i=1$ implies that $[j,i,j]$ is an $\cF$-path, for some $j\neq i$. We can apply Lemma \ref{lemma:not3paths} to get \eqref{eq2:lemmacv}.
\end{proof}

We already have proven Theorem \ref{teo:01}, as shown below.
\begin{proof}[Proof of Theorem \ref{teo:01}]
    Proposition \ref{prop1} and Lemma \ref{lemma:cV} imply the first part of the theorem. Item (i) follows from Lemma \ref{lemma:ui=ouui=0} and Theorem \ref{teo:000000004}. The second part of Theorem \ref{teo:000000004} implies Item (ii). Item (iii) was already established in Corollary \ref{cor:constants}.
\end{proof}

\section{Proximal action}\label{sec:prosimal}

An important concept in the study of RDSs is \emph{proximality}. Introduced by Furstenberg in his seminal paper \cite{Fur:67}, it has since been employed by many researchers in the field. Proximality measures the degree of similarity between two orbits in a dynamical system. Specifically, two points in a dynamical system are considered proximal if their distance remains small over time. See item (e) of Theorem \ref{teo3} for further details.

Let us state the main result of this section,  which further illustrates the usefulness of the collections of intervals defined in Theorem \ref{teo:01}.
\begin{theorem}\label{teo3}
   Let $(\cF,\nu)$ be an RDS of circle homeomorphisms without finite orbits. Let $d$ be the number of ergodic stationary measures.  Also, assume $d>1$. For $i\in\{1,\ldots,d\}$, consider the finite collection of disjoint closed intervals $\cA_i$ as in Theorem \ref{teo:000000004}. Then, there exists a finite collection of disjoint closed intervals $\cB_i$ such that
    \begin{itemize}
        \item[(a)] $\# \cB_i=\# \cA_i$ or $\#\cB_i=2\#\cA_i$;
        \item[(b)] $K_i\subset \cup_{I\in\cB_i} I\subset \cup_{I\in\cA_i} I$;
        \item[(c)] each interval in $\cA_i$ contains either one or two intervals in $\cB_i$;
        \item[(d)] for all $I,J\in\cB_i$ there exists $f\in\cF$ such that $f(I)\subset J$;
        \item[(e)] the semigroup $G_\cF$ generated by $\cF$ acts proximally on each $I\in\cB_i$, i.e., for every $x,y\in I$ there exists a sequence $(g_n)_{n\in\bN}$ in $G_\cF$ such that
        \[
        \lim_{n\to\infty}\mbox{dist}(g_n(x),g_n(y))=0.
        \]
    \end{itemize}
\end{theorem}

Theorem \ref{teo3} is a "probabilistic" formulation of Theorem \ref{newteo:1}. In fact, if we consider for each minimal $K_i$ the collection $\cE_i=\cE_{K_i}$ (defined in Section \ref{sec:topprop}) instead of $\cA_i$, we obtain the same result on the same consequences. 
\begin{proof}[Proof of Theorem \ref{newteo:1}]
    Proceed analogously as in the proof of Theorem \ref{teo3}, considering collections $\cE_i=\cE_{K_i}$ instead of collections $\cA_i$. Then we have a new collection $\hat \cE_i$ with the properties of $\cB_i$, and such that 
    \[
    K_i\subset \cup_{I\in\hat\cE_i} I\subset \cup_{I\in\cB_i} I.
    \]
    To conclude, consider $L_i=\cup_{I\in\hat\cE_i} I$.
\end{proof}

\begin{remark}
    Consider an RDS $(\cF,\nu)$ as in Theorem \ref{teo3}.
    Since $d>1$ and there are no finite orbits, we must have that local contraction property. Note that local contraction property is preserved for any subsystem. 
     Applying \cite[Proposition 4.18]{Mal:17} and using the finiteness of collection $\cB_i$, we can conclude that the action of $G_\cF$ is synchronizing on each interval $I\in\cB_i$, i.e., for every $x,y\in I$ and $\nu^\bN$-almost every sequence $\omega\in\cF^\bN$ we have
        \[
        \lim_{n\to\infty}\mbox{dist}(f_\omega^n(x),f_\omega^n(y))=0.
        \]
\end{remark}

\begin{corollary}\label{corollary:teo2}
   Assume that $(\cF,\nu)$ is an RDS without finite orbits, with $\cF\subset\mbox{Hom}^+(\bS^1)$. Also, assume that the number $d$ of ergodic stationary measures is greater than 1. For $i\in\{1,\ldots,d\}$, consider the finite collection of disjoint closed intervals $\cA_i$ as in Theorem \ref{teo:000000004}.  Then, the semigroup $G_\cF$ generated by $\cF$ acts proximally on each $I\in\cA_i$.
\end{corollary}

In the proof of Theorem \ref{teo3}, it becomes apparent that the primary elements utilized are the topological properties of the collections $\cA_i$. Specifically, for a fixed $i\in\{1,\ldots,d\}$, we employ the $\cF$-invariance properties obtained in Theorem \ref{teo:000000004}. Furthermore, we use the fact that almost every random orbit starting from an extremal point of any interval within $\cA_i$ converges to the minimal $K_i$. Consequently, in the case $d=1$, when the closed minimal set $K$ is contained within a finite union of nontrivial closed intervals, we can follow the same approach as in the proof of Theorem \ref{teo3} to establish the subsequent result.

\begin{corollary}\label{corollary2:teo2}
   Let $(\cF,\nu)$ be an RDS of circle homeomorphisms without finite orbits. Let the number $d$ be the number of ergodic stationary measures and assume $d=1$. Let $K\subset\bS^1$ be the closed minimal set.
   If there exists a finite collection $\cA$ of closed intervals such that $\cup_{I\in\cA}I\subsetneq\bS^1$ and for any $I,J\in\cA$, $f\in \cA$ there exist $I',J'\in\cF$ satisfying $f(I)\subset I'$ and $f(J')\subset J$. Then, $K\subset \cup_{I\in\cA}I$ and there exists a finite collection of disjoint closed intervals $\cB$ such that
    \begin{itemize}
        \item[(a)] $\# \cB=\# \cA$ or $\#\cB=2\#\cA$;
        \item[(b)] $K\subset \cup_{I\in\cB} I\subset \cup_{I\in\cA} I$;
        \item[(c)] each interval in $\cA$ contains either one or two intervals in $\cB$;
        \item[(d)] for all $I,J\in\cB$ there exists $f\in\cF$ such that $f(I)\subset J$;
        \item[(e)] the semigroup $G_\cF$ generated by $\cF$ acts proximally on each $I\in\cB$.
    \end{itemize}
\end{corollary}

Before showing Theorem \ref{teo3}, let us establish some useful results. In what follows, as $x,y\in [a,b]$, $x\neq y$ we write $x<y$ whenever $\langle a,x,y,b\rangle\in\CO$.

\begin{lemma}\label{lemma:forTeo2}
    Assume that $(\cF,\nu)$ is an RDS with $\cF\subset\mbox{Hom}(\bS^1 )$ without finite orbits. Let the number $d$ be the number of ergodic stationary measures and assume $d\geq 2$. Let $G_\cF$ be the semigroup generated by $\cF$. Fix $i\in\{1,\ldots,d\}$. Let $\cA_i$ be as in Theorem \ref{teo:000000004}. Given $I=[a,b]\in\cA_i$, let $S_I\subset G_{\cF}$ denote the semigroup of the maps $f$ in $G_{\cF}$ such that $f(I)\subset I$. Let $S_I^+\subset S_I$ be the semigroup of preserving-orientation maps in $S_I$. Set
    \[
    \hat a=\sup\{f(a)\colon f\in S_I^+\},
    \]
    and 
    \[
    \hat b=\inf\{f(b)\colon f\in S_I^+\}.
    \]
    Then, 
    \begin{itemize}
        \item[(i)] either $\hat a\geq\hat b$ and $G_\cF$  acts proximally on $[a,b]$ , 
        \item[(ii)] or,  $\hat a<\hat b$ and $G_\cF$ acts proximally on $[a,\hat a]$ and $[\hat b,b]$. In this case, we also have $\hat a,\hat b\in K_i$ and $(\hat a ,\hat b)\cap K_i=\emptyset$
    \end{itemize}
    In case (ii), for each $f\in S_I$ we have
    \begin{itemize}
        \item[(iii)] either $f$ preserves orientation, $f([a,\hat a])\subset [a,\hat a]$, and $f([\hat b, b])\subset [\hat b, b]$,
        \item[(iv)] or, $f$ does not preserve orientation, $f([a,\hat a])\subset [\hat b, b]$, and $f([\hat b, b])\subset [a,\hat a]$.
    \end{itemize}
     
\end{lemma}
    \begin{proof}
    Using that $S_I^+$ is a semigroup, we get that for $g\in S_I^+$, 
    \begin{align*}
        g(\hat a)&= g(\sup\{f(a)\colon f\in S_I^+\})\\
        &=\sup\{g(f(a))\colon f\in S_I^+\}\\
        &\leq \sup\{f(a)\colon f\in S_I^+\}=\hat a,
    \end{align*}
    that is $g(\hat a)\leq \hat a$. Analogously, for $g\in S_I^+$ we get $g(\hat b)\geq \hat b$.
    
    \textbf{Case $\hat a>\hat b$.} In this case, it is easy show the proximal action of $G_\cF$ on $I$. In fact, take $g\in S_I^+$ such that $\hat b<g(a)\leq g(b)\leq b$. Then consider consider a sequence $(g_n)_{n\in\bN}$ in $ S_I^+$ such that $g_n(b) $ converges to $\hat b$ as $n\to\infty$. Hence, for every $x\in I$, $\hat b\leq g_n(\hat b)\leq g_n(g(a))\leq g_n(g(x)))\leq g_n(g(b))\leq g_n(b)$ and so the sequence $(g_n(g(x)))_{n\in\bN}$ converges to $\hat b$. So that the sequence $(g_n\circ g)_{n\in\bN}$ in $S_I^+\subset G_\cF$ degenerates the interval $I$ to $\{\hat b\}$, which implies what is desired.  
    
    \textbf{Case $\hat a=\hat b$.} For all $g\in S_I^+$ we get 
    \[
    g(\hat b)\geq \hat b=\hat a\geq g(\hat a),
    \]
    i.e., $c=\hat a=\hat b\in [a,b]$ is a fixed point by every map in $S_I^+$. Now, let us show that for any $\varepsilon>0$, there exists $g\in S_I^+$ such that $g([a,b])\subset (c-\varepsilon,c+\varepsilon)$. By definition of supremum and infimum, given $\varepsilon>0$ there exist $f\in S_I^+$ and a sequence $(f_n)_{n\in\bN}$ such that 
    \[
    c\leq f(b)<c+\varepsilon,
    \]
    and 
    \[
    \lim_{n\to\infty}f_n(a)=c. 
    \]
    By continuity,
    \[
    \lim_{n\to\infty}f\circ f_n(a)=f(c)=c.
    \]
    By definition of $c$, it is clear that $c\leq f_n(b)\leq b$ for all $n\in\bN$. So that for $N$ large enough we have 
    \[
    f\circ f_N(a)\in (c-\varepsilon,c]\quad\mbox{and}\quad f\circ f_N(b)\in [c,f(b)]\subset [c,c+\varepsilon).
    \]
    Consequently, $g([a,b])\subset (c-\varepsilon,c+\varepsilon)$ for $g=f\circ f_N$. Using the arbitrariness of $\varepsilon$, we can conclude that the semigroup $S_I^+$ (and so $G_{\cF}$) acts proximally on $I=[a,b]$.

     \textbf{Case $\hat a<\hat b$.} By definition of $\hat a$ and $\hat b$, it is clear that $S_I^+$ acts proximally on $[a,\hat a]$ and $[\hat b,b]$. Since $S_I^+\subset S_I\subset G_\cF$, we have that $G_\cF$ acts proximally on $[a,\hat a]$ and $[\hat b,b]$. Now, let us show that $f( a),f( b)\in [a,\hat a]\cup[\hat b,b]$ for all $f\in S_I$. Take $f\in S_I$ and note $f^2\in S_I^+$. If $f$ is a not preserving orientation map then $f(a)>\hat b$, because
     \[
     a\leq f(b) <f(a)<\hat b,
     \]
     implies
     \[
     f(\hat b)<f^2(a) <f^2(b)\leq f(a)<\hat b
     \]
     which is a contradiction to the definition of $\hat b$. Analogously, $f(b)<\hat a$. Therefore, for $f\in S_I$ then 
     \[
     f(a)\in [a,\hat a],\, f(b)\in[\hat b,b]
     \]
     whenever $f$ preserves orientation, or
     \[
     f(a)\in[\hat b,b],\, f(b)\in [a,\hat a]
     \]
     whenever $f$ does not preserve orientation. Using Theorem \ref{teo:000000004}, for $f\in G_\cF$ such that $f(a)\in [a,b]$ we conclude that $f\in S_I$. Since $u_i(a)=1$, for any $x\in K_i$ we can find $f\in G_\cF$ such that $f(a)$ is arbitrarily close to $x$.
    Hence, we must have $(\hat a,\hat b)\cap K_i=\emptyset$. 
    
    Moreover, $\hat a,\hat b\in K_i$. Let us show that $\hat a\in K_i$, the proof of $\hat b\in K_i$ is analogous. Use the finiteness of the collection $\cA_i$, the definition of the weight map $u_i$ and that $u_i(a)=u_i(b)=1$ to conclude $K_i\cap [a,\hat a]\neq \emptyset$. Hence, if $\hat a\notin K_i$, then take $x\in [a,\hat a)\cap K_i$ such that $(x,\hat a)\cap K_i=\emptyset$. Now, by $\cF$-invariance of $K_i$, for all $f\in S_I^{+}$, $f(x)\in K_i$ and $f(a)\leq f(x)\leq \hat a$. But, by definition of $x$, $f(x)\leq x$. So that $f(a)\leq x<\hat a$ for all $f\in S_I^+$, which contradicts the definition of $\hat a.$ 

The Lemma is proven.
    \end{proof}

    \begin{remark}\label{rem1:coro:teo2}
        Case (ii) in Lemma \ref{lemma:forTeo2} can occur only when there exists a map in $S_I$ that does not preserve orientation.
        \end{remark}

\begin{proof}[Proof of Theorem \ref{teo3}]
Consider $i\in\{1,\ldots,d\}$ and $\cA_i=\{I_1,\ldots,I_m\}$ as in Theorem \ref{teo:000000004}. Given $I_j=[a_j,b_j]\in \cA_i$, consider $A_j=[a_j,\hat a_j]$ and $B_j=[\hat b_j,b_j]$ for $\hat a_j$ and $\hat b_j$ are defined analogously to $\hat a$ and $\hat b$ in Lemma \ref{lemma:forTeo2} taking $a=a_j$ and $b=b_j$. Case $\hat a_j\geq\hat b_j$, set $\cB_{i,j}=\{I_j\}$. Case $\hat a_j<\hat b_j$, set $\cB_{i,j}=\{A_j,B_j\}$. Consider
\[
\cB_i=\cup_{j=1}^m\cB_{i,j}.
\]
By definition of $\cB_i$, item (c) holds. Items (b) and (e) are consequences of items (i) and (ii) of Lemma \ref{lemma:forTeo2}. Item (d) is a consequence from item (b) of Theorem \ref{teo:000000004} and (iii), (iv) of Lemma \ref{lemma:forTeo2}. 

To see item (a) it is enough to show that $\# \cB_{i,j}=\# \cB_{i,k}$ for all $j,k\in\{1,\ldots,m\}$. Assume that there exist $j,k$ such that $\# \cB_{i,j}=2$ and $\# \cB_{i,k}=1$.  Since $\# \cB_{i,j}=2$, we have $\hat a_j<\hat b_j$. Set $\delta=\mbox{dist}(\hat a_j, \hat b_j)>0$. Fix $x\in [a_j,\hat a_j]\subset [a_j,b_j]$ and $y\in [\hat b_j,b_j]\subset [a_j,b_j]$. By item (iv) of Lemma \ref{lemma:forTeo2}, we get $\mbox{dist}(g(x),g(y))\geq \delta>0$ for all $g$ preserving $I_j$. By item (b) of Theorem \ref{teo:000000004}, there exists $f\in G_\cF$ such that $f(I_j)\subset I_k$. From the proximal action of $G_\cF$ on $I_k$, we can find a sequence $(g_n)_{n\in\bN}$ in $G_\cF$ such that 
\[
\lim_{n\to\infty}\mbox{dist}(g_n(f(x)),g_n(f(y)))=0.
\]
Without loss of generality, let us assume that the sequences $(g_n(f(x)))_{n\in\bN}$ and $(g_n(f(y)))_{n\in\bN}$ converge to a point $z\in I_r$, for some $r\in\{1,\ldots,m\}$. Take $h\in\cF$ such that $h(I_r)\subset I_j$. Consider $V_z=h^{-1}(h(z)-\varepsilon,h(z)+\varepsilon)$ for some $\varepsilon\in(0,\delta/4)$. Take $N$ such that $g_N(f(x)),g_N(f(y))\in V_z$. Hence, $g=h\circ g_N\circ f$ preserves $I_j$. Further, $g(x),g(y)\in (h(z)-\varepsilon,h(z)+\varepsilon)$ which is a contradiction because $\mbox{dist}(g(x),g(y))>\delta>2\varepsilon$. Therefore, there is no $i,j$ such that $\# \cB_{i,j}\neq\# \cB_{i,k}$, which implies the desired.
\end{proof}

\begin{proof}[Proof of Corollary \ref{corollary:teo2}]
    Following the notation in the proof of Theorem \ref{teo3}. By Remark \ref{rem1:coro:teo2}, for all $i$ and $j$ we get that $\cB_{i,j}=\{I_j\}$. Therefore, for all $i=1,\ldots,d$ we have $\cB_i=\cA_i$. Applying Theorem \ref{teo3}, we conclude the desired.
\end{proof}

\section{The inverse RDS}\label{sec:d+ and d-}

Assume that $(\cF,\nu)$ is an RDS with $\cF\subset\mbox{Hom}(\bS^1 )$. Consider the \emph{inverse RDS} $(\cF^{-},\nu^{-})$ associated to $(\cF,\nu)$, i.e.,
\[
\cF^-= \{f^{-1}\colon f\in\cF\},
\]
and
\[
\nu^-(\cdot)= \nu(\{f\colon f^{-1}\in \cdot\}).
\]
Also, we assume there is no finite orbit for $\cF $. Hence, no finite orbit exists for $\cF^{-}$.

Let us begin the discussion on the number of stationary measures of $(\cF,\nu)$ and $(\cF^{-},\nu^{-})$ with an example illustrating that such quantities do not necessarily coincide.
\begin{figure}[h]
    \centering
    \includegraphics[width=0.5\textwidth]{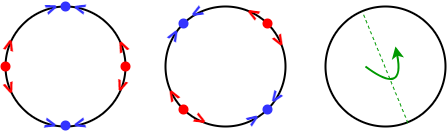}
    \caption{Dynamics on the circle of the functions $f_1$, $f_2$ and $f_3$ in Example \ref{ex1} respectively.}
    \label{fig:ex-1}
\end{figure}
\begin{example}\label{ex1}
Consider $\cF=\{f_1,f_2,f_3\}\subset\mbox{Hom}(\bS^1 )$ where the maps $f_i$'s are defined as follow:
\begin{itemize}
    \item $f_1$ is a preserving orientation homeomorphism with only 4 fixed points $0$, $1/4$, $1/2$ and $3/4$. Where $0$ and $1/2$ are attracting points; and, $1/4$ and $3/4$ are repelling points.
    \item $f_2$ is a preserving orientation homeomorphism with only 4 fixed points $1/8$, $3/8$, $5/8$ and $7/8$. Where $3/8$ and $7/8$ are attracting points; and, $1/8$ and $5/8$ are repelling points.
    \item $f_3$ is a reflection map on the line that passes through the points $7/16$ and $15/16$. So $f_3$ is an isometry fixing only $7/16$ and $15/16$. Hence, the inverse map $f_3^{-1}$ of  $f_3$ is $f_3^2=f_3\circ f_3$. Further, $f_3(0)=7/8$, $f_3(1/4)=5/8$, $f_3(1/2)=3/8$ and $f_3(3/4)=1/8$.
\end{itemize}
We are consider $\bS^1=[0,1)$.
Since each point fixed by $f_1$ is different from any point fixed by $f_2$, there is no probability measure $m$ on $\bS^1$ such that $f_{\ast}m=m$ for all $f\in \cF$ (and so for all $f\in\cF^-$). Note that
$f_i([3/4,1/8])\subset [3/4,1/8]$ and $f_i([1/4,5,8])\subset [1/4,5,8]$, for all $i=1,2,3$. Therefore, it is not difficult to conclude that there are exactly $d_+=2$ invariant minimal sets for $\cF$. On the other hand, using the dynamics of the maps in $\cF$, we have that $\cF^-=\{f_1^{-1},f_2^{-1},f_3^{-}\}$ has a unique minimal set contained in $[1/2,7/8]\cup[0,3/8]$. Let $d_-=1$. Then, For any probability measure $\nu$ supported in $\cF$ (i.e., with $\nu(\{f_i\})>1$ for $i=1,2,3$), the number of ergodic stationary measures for $(\cF,\nu)$ is $d_+=2$. However, if we consider $\nu^-$ by $\nu^-(\{f_i^{-1}\})=\nu(\{f_i\})$, we have that the number of ergodic stationary measures for $(\cF^{-1},\nu^{-1})$ is $d_-=1$.
\end{example}


Let $d_+$ and $d_-$ denote the number of ergodic stationary measures for $(\cF,\nu)$ and $(\cF^-,\nu^-)$ respectively.
To relate the quantities $d_+$ and $d_-$, we will use Theorems \ref{teo2} and \ref{teo3}. 
First, let us establish some preliminary results.

\begin{lemma}\label{lemma:Cij(n+2)}
    Let $(\cF,\nu)$ be an RDS of circle homeomorphisms without finite orbits. Let $(\cF^{-},\nu^{-})$ be the inverse RDS. Let $d_+$ and $d_-$ be the numbers of ergodic stationary measures associated to $(\cF,\nu)$ and $(\cF^{-},\nu^{-})$ respectively. If $d_+=n+2$ then $d_-\geq n+1$.
\end{lemma}
    Before proving Lemma \ref{lemma:Cij(n+2)}, let us introduce some useful sets. For $i,j\in\{1,\ldots,d_+\}$, $i\neq j$, let $C_{i,j}$ be the set of closed intervals $I$ with one of its extremal points in $\cup_{J\in \cA_i}J$ and the other extremal point in $\cup_{J\in \cA_j}J$, such that the interior of $I$ h does not intersect $\cup_{J\in \cup_k\cA_k}J$. Note that $C_{i,j}=C_{j,i}$ and that it is a finite set. Furthermore, $C_{i,j}\neq\emptyset$ whenever $i$ and $j$ are neighbors. 

\begin{remark}\label{remark:Cij(n+2)}
By (b) in Theorem \ref{teo:000000004}, for all $I,J\in C_{i,j}$ there exists $f\in\cF$ such that $f^{-1}(I)\subset J$ (or, $I\subset f(J)$). In other words, the $\cF$-invariance of the sets $\cA_i$'s implies the $\cF^{-1}$-invariance of the sets $C_{i,j}$'s. Hence,
\[
f^{-1}\left( \cup_{I\in C_{i,j}}I\right)\subset \cup_{I\in C_{i,j}}I,
\]
for all $f\in G_\cF$. Then, for each $(i,j)\in\cV$, $i\neq j$, there exists a closed $\cF^-$-minimal set $F\subset C_{i,j}$.
\end{remark}

\begin{proof}[Proof of Lemma \ref{lemma:Cij(n+2)}]

Assume $d_+=n+2$. By Lemma \ref{lemma:cV}, then 
\[
\#\{i\in \{1,\ldots,d_+\}\colon \#V_i=2\}\geq n.
\]
Hence, there are $n$ distinct points $i_1,\ldots,i_n,$ in \(\{1,\ldots,d_+\}\) such that $V_{i_k}=2$, for $k=1,\ldots,n$. Since $d_+=n+2>n$, we have that any $n$-cycle \footnote{A $n$-cycle is a path of the form $[j_0,j_1,\ldots,j_n]$ with $j_0=j_n$ and $j_k\neq j_{\ell}$ for $k,\ell=1,\ldots,n$ and $\ell\neq k$.} is not a $\cF$-path. Therefore, we can find $i_{n+1},j_1,\ldots,j_n,j_{n+1}\in\{1,\ldots,d_+\}$ such that $(i_1,j_1),\ldots,(i_{n+1},j_{n+1})\in\cV$ such that for $k\neq \ell$
\[
i_k\neq j_k\quad \mbox{and}\quad\{i_k,j_k\}\neq \{i_\ell,j_\ell\}.
\]
For $k=1,\ldots,n+1$, consider the set $C_{i_k,j_k}$ as defined above Remark \ref{remark:Cij(n+2)}.
Let $C_k=C_{i_k,j_k}$. Note that for $k\neq \ell$, $C_k\neq C_\ell$ and so
\[
\left(\cup_{I\in C_k}I\right)\cap\left(\cup_{J\in C_\ell}J\right)=\emptyset.
\]
By Remark \ref{remark:Cij(n+2)}, for $k=1,\ldots,n+1$, there exists a closed $\cF^-$-minimal set $F_k\subset C_k$. Therefore, there exist at least $n+1$ closed $\cF^-$-minimal sets and so $d_-\geq n+1$.     
\end{proof}

\begin{lemma}\label{lemma:d+=d-}
    Let $(\cF,\nu)$ be an RDS of circle homeomorphisms without finite orbits. Let $(\cF^{-},\nu^{-})$ be the inverse RDS. Let $d_+$ and $d_-$ be the numbers of ergodic stationary measures associated to $(\cF,\nu)$ and $(\cF^{-},\nu^{-})$ respectively. If $\cF\subset\mbox{Hom}^+(\bS^1 )$, then $d_-\geq d_+$.
\end{lemma}
    Let us introduce a variation of the sets $C_{i,j}$'s, now we will respect the order of the indices. For $i,j\in\{1,\ldots,d_+\}$, $i\neq j$, let $C_{[i,j]}$ be the set of closed intervals $I=[a,b]$ with $a\in\cup_{J\in \cA_i}J$ and $b\in\cup_{J\in \cA_j}J$, such that $(a,b)$ has no points of $\cup_{J\in \cup_k\cA_k}J$. Note that $C_{[i,j]}\neq C_{[j,i]}$ (provided that one of the two is non-empty) and $C_{[i,j]}$ is a finite set. Further, $C_{[i,j]}\neq\emptyset$ whenever $[i,j]$ is $\cF$-path. By (b) in Theorem \ref{teo:000000004}, 
\[
f^{-1}\left( \cup_{I\in C_{[i,j]}}I\right)\subset \cup_{I\in C_{[i,j]}}I,
\]
for all $f\in G_\cF\subset\mbox{Hom}^{+}(\bS^1 )$. Hence, for each $\cF$-path $[i,j]$ there exists a closed $\cF^-$-minimal set $F\subset C_{[i,j]}$.

\begin{proof}[Proof of Lemma \ref{lemma:d+=d-}]

If $d_+=1$, it is clear that $d_-\geq d_+$. Assume $d_+\geq 2$. For all $i=1,\ldots,d_+$, take $j_i\in\{1,\ldots,d_+\}\backslash\{i\} $ such that $[i,j_i]$ is an $\cF$-path and so $(i,j_i)\in\cV$. As we discussed above, we have for each $\cF$-path $[i,j_i]\in\cV$ there exists a closed $\cF^-$-minimal set $F_i\subset C_{[i,j_i]}$. Hence, there exist at least $d_+$ closed $\cF^-$-minimal sets and so $d_-\geq d_+$.
\end{proof}

Now we are ready to prove Theorem \ref{teo2}.
\begin{proof}[Proof of Theorem \ref{teo2}]
By Lemma \ref{lemma:Cij(n+2)}, $d_-\geq d_+-1$. Using that $(\cF^-)^-=\cF$, we also get $d_+\geq d_--1$. Consequently, $d_++1\geq d_-\geq d_+-1$ and hence $|d_+-d_-|\leq 1$.

Now, assume $\cF\subset\mbox{Hom}^+(\bS^1 )$. Since $(\cF^-)^-=\cF$, by Lemma \ref{lemma:d+=d-}, $d_-\geq d_+$ and $d_+\geq d_-$. Therefore $d_-=d_+$.    
\end{proof}

\bibliographystyle{alpha}
\bibliography{bib}
\end{document}